\documentclass[12pt,a4paper,leqno]{amsart}

\usepackage{amsfonts,amssymb,mathrsfs,bm,enumerate,graphicx}
\usepackage[pdftex]{hyperref}
\usepackage[numbers]{natbib}

\numberwithin{equation}{section}
\theoremstyle{plain}
 \newtheorem{thm}{Theorem}[section]
 \newtheorem{lem}[thm]{Lemma}
 \newtheorem{cor}[thm]{Corollary}
 \newtheorem{prop}[thm]{Proposition}
 
\theoremstyle{definition}
 \newtheorem{defn}[thm]{Definition}
 \newtheorem{ex}[thm]{Example}
 \newtheorem{rem}[thm]{Remark}

\newcommand{\ep}{\varepsilon}

\newcommand{\q}{\quad}

\newcommand{\eqd}{\overset{\mathrm d}{=}}
\newcommand{\sek}{\int_0^{\infty}}

\newcommand{\wh}{\widehat}

\newcommand{\la}{\langle}
\newcommand{\ra}{\rangle}
\newcommand{\B}{\mathcal{B}}

\newcommand{\R}{\mathbb{R}}
\newcommand{\N}{\mathbb{N}}
\newcommand{\Z}{\mathbb{Z}}
\newcommand{\Q}{\mathbb{Q}}

\newcommand{\rd}{{\mathbb R^d}}
\newcommand{\law}{\mathcal L}

\newcommand{\n}{\noindent}

\def\1{\scalebox{0.94}{$1$}\hspace{-0.36em}1}

\allowdisplaybreaks[4]
\begin{document}
\setlength{\baselineskip}{18pt}
\setlength{\parindent}{1.8pc}

\n
{\Large \bf The dichotomy of recurrence and transience of\\ semi-L\'evy processes}

\vskip 5mm
\n
({\it Running title:} Recurrence and transience of semi-L\'evy processes)

\vskip 10mm
\n
{MAKOTO MAEJIMA$^{*}$, TAISUKE TAKAMUNE and YOHEI UEDA}

\vskip 5mm
\n
{Department of Mathematics, Keio University, 3-14-1, Hiyoshi, 
Kohoku-ku, Yokohama 223-8522, Japan.}

\vskip 10mm
\n
Semi-L\'evy process is an additive process with periodically stationary increments.
In particular, it is a generalization of L\'evy process.
The dichotomy of recurrence and transience of L\'evy processes is well known,
but this is not necessarily true for general additive processes.
In this paper, we prove the recurrence and transience dichotomy of semi-L\'evy processes.
For the proof, we introduce a concept of semi-random walk and
discuss its recurrence and transience properties.
An example of semi-L\'evy process constructed from two independent L\'evy processes is investigated.
Finally, we prove the laws of large numbers for semi-L\'evy processes.

\vskip 5mm
\n
{\it keywords:} {L\'evy process; semi-L\'evy process; recurrence; transience; semi-random walk; 
law of large numbers}
\section{Introduction}
Recurrence and transience problem has been studied for many stochastic processes.
Among those, the dichotomy of recurrence and transience of L\'evy processes is well known
(see, e.g., Section 35 of \citet{Sato's_book1999}), which is not necessarily true for
general additive processes (see Remark 3.4 of \citet{SatoYamamuro1998}).
In this paper, we show the dichotomy for semi-L\'evy processes.

A stochastic process $\{ X_t\}_{t\ge 0}$ on $\rd$ is called an additive process if $X_0=0$ a.s., 
it is stochastically continuous, it has independent increments and its sample paths
are right-continuous in $t\ge 0$ and have left-limits in $t>0$.
Further, if $\{ X_t\}_{t\ge 0}$ has stationary increments, it is a L\'evy process.
As an extension of L\'evy process, we define a subclass of
additive processes with the property that for some $p>0$,
$$
X_{t+p}- X_{s+p} \eqd X_t-X_s \q \text{for any}\,\, s,t\ge 0.
$$
Process with this property is called a semi-L\'evy process with period $p$. 
Semi-L\'evy processes have already been considered in the literature by
\citet{MaejimaSato2003} and in Lemma 3.3 of \citet{BeckerKern2004}.

We give several comments regarding semi-L\'evy processes.
An infinitely divisible distribution $\mu$ on $\rd$ is called selfdecomposable if for any
$b>1$ there exists a distribution (necessarily infinitely divisible) $\mu_b$ on $\rd$
such that $\wh\mu(z) = \wh \mu (b^{-1}z)\wh \mu_b(z)$, (where $\wh\mu$ is the characteristic function 
of $\mu$), and it is called semi-selfdecomposable with span $b>1$ 
if there exists an infinitely divisible
distribution $\rho$ on $\rd$ such that $\wh\mu(z) = \wh\mu(b^{-1}z)\wh\rho(z)$.
L\'evy processes are connected to selfdecomposable distributions in the sense that $\mu$ is 
selfdecomposable if and only if  there exists a L\'evy process $\{ X_t\}_{t\ge 0}$ 
with $E[\log (1+|X_1|)]<\infty$ satisfying
\begin{equation}\label{stochastic int rep}
\mu = \law \left ( \sek e^{-t}dX_t\right ),
\end{equation}
where $\law(Y)$ denotes the law of a random variable $Y$.
In the same fashion, semi-L\'evy processes are related to semi-selfdecomposable distributions in the
sense that $\mu$ is 
semi-selfdecomposable with span $b>1$ if and only if there exists a semi-L\'evy process 
$\{ X_t\}_{t\ge 0}$
with period $p=\log b$ which is a semimartingale with $E[\log (1+|X_1|)]<\infty$ and 
which satisfies \eqref{stochastic int rep}
(see \citet{MaejimaSato2003}).

Furthermore, as shown in \citet{MaejimaSato2003} and \citet{BeckerKern2004},
semi-L\'evy processes are also related to semi-selfsimilar additive processes
which were introduced and deeply studied in \citet{MaejimaSato1999}.
More precisely, semi-L\'evy processes and semi-selfsimilar additive processes 
can mutually be represented by stochastic integrals with respect to each other
if they are semimartingales.

\begin{ex}\label{example}
The following example of semi-L\'evy process was given by Ken-iti Sato 
to the authors through a private communication.
Let $\{Y_t\}_{t\ge 0}$ and $\{Z_t\}_{t\ge 0}$ be two independent L\'evy processes and let
$0<q<p$ be arbitrary.
A stochastic process $\{X_t\}_{t\ge 0}$ defined by
\begin{align*}
& X_0 = 0 \quad a.s., \\
& X_t = 
\begin{cases}
X_{np} + Y_t - Y_{np} & (np < t \le np+q),\\
X_{np+q} + Z_t - Z_{np+q} & (np+q < t \le (n+1)p)
\end{cases}
\end{align*}
is a semi-L\'evy process with period $p$.
The recurrence and transience property of this process is discussed in Section 4.
\end{ex}

Throughout this paper, $|x|$ denotes the Euclidean norm of $x\in\rd$,
$\Z_+$ denotes the set of nonnegative integers and
$B_a := \{x \in \mathbb{R}^d\colon |x|<a \}$ denotes the open ball of radius $a>0$ around the origin.
\begin{defn}
[\citet{SatoYamamuro1998}]
Let $s\geq 0$.
An additive process $\{X_t\}_{t\ge 0}$ on $\rd$ is called \emph{$s$-recurrent} if
\begin{equation*}
\liminf_{t \to \infty} |X_t-X_s| = 0 \quad \text{a.s.},
\end{equation*}
and it is called \emph{recurrent} if it is $s$-recurrent for any $s\geq 0$.
It is called \emph{transient} if
\begin{equation*}
\lim_{t \to \infty} |X_t| = \infty \quad \text{a.s.}
\end{equation*}
\end{defn}

\begin{rem}
Note that, for L\'evy processes, by the stationary increment property, 
recurrence and $0$-recurrence are equivalent.
This is also the case for semi-L\'evy processes as seen in Theorem \ref{sLdich} (i) below. 
\end{rem}

Now, we state the main results of this paper.

\begin{thm}
\label{sLdich}
Let $\{X_t\}_{t\ge 0}$ be a semi-L\'evy process on $\rd$ with period $p>0$. 
Then we have the following.
\begin{enumerate}[\quad \rm (i)]
\item It is recurrent if it is $0$-recurrent.
\item It is either recurrent or transient.
\item It is recurrent if and only if
\begin{equation}
\int^\infty_0 P(X_t \in B_a)dt = \infty \quad \text{for every }a>0. \label{sLii}
\end{equation}
\item It is recurrent if and only if
\begin{equation*}
\int^\infty_0 \1_{B_a}(X_t)dt = \infty \quad \text{a.s.} \quad \text{for every} \;\, a>0. \label{sLiii}
\end{equation*}
\item It is transient if and only if
\begin{equation}
\int^\infty_0 P(X_t \in B_a)dt < \infty \quad \text{for every }a>0. \label{sLiv}
\end{equation}
\item It is transient if and only if
\begin{equation}
\int^\infty_0 \1_{B_a}(X_t)dt < \infty \quad \text{a.s.} \quad \text{for every } a>0. \label{sLv}
\end{equation}
\item Fix $h\in (0,\infty)\cap p\Q$ arbitrarily.
Then, the semi-L\'evy process $\{X_t\}_{t\ge 0}$ is recurrent if and only if the semi-random walk
$\{X_{nh}\}_{n\in\Z_+}$ is recurrent.
$($See Section $2$ for the definition of semi-random walk.$)$
\end{enumerate}
\end{thm}

Further, let $\{X_t\}_{t\ge 0}$ be a semi-L\'evy process with period $p>0$.
Since $X_p$ is an infinitely divisible random variable, 
then there is a L\'evy process $\{Y_t\}_{t\ge 0}$ such that $Y_1\eqd X_p$. 
Now, as a consequence of Theorem \ref{sLdich}, we deduce the following.

\begin{thm}\label{LevsemiLev}
$\{X_t\}_{t\ge 0}$ is recurrent if and only if $\{Y_t\}_{t\ge 0}$ is recurrent.
\end{thm}

Finally, as a consequence of Theorem \ref{LevsemiLev}, the criterion of 
Chung-Fuchs type immediately follows. 

\begin{cor}
Let $\{ X_t\}_{t\ge 0}$ be a semi-L\'evy process on $\rd$ 
with period $p>0$ and let $\psi(z) = \log \wh{\law(X_p)}(z)$.
Fix $a>0$.
Then the following three statements are equivalent.
\begin{enumerate}[\quad \rm (i)]
\item $\{X_t\}_{t\ge 0}$ is recurrent.
\item $\lim_{q\downarrow 0}\int_{B_a} {\rm Re}\left (\frac{1}{q-\psi(z)}\right ) dz =\infty.$
\item $\limsup_{q\downarrow 0}\int_{B_a} {\rm Re}\left (\frac{1}{q-\psi(z)}\right ) dz =\infty.$
\end{enumerate}
\end{cor}
\begin{proof}
The claim of this corollary follows from the corresponding result for 
L\'evy processes (see, e.g., Theorem 37.5 of 
\citet{Sato's_book1999}) and Theorem \ref{LevsemiLev}.
\end{proof}

This paper is organized as follows.
In Section 2, we introduce the notion of semi-random walk and prove its recurrence and
transience dichotomy. 
In Section 3, we give proofs of Theorems \ref{sLdich} and \ref{LevsemiLev}.
In Section 4, we investigate the recurrence and transience property of the
semi-L\'evy process defined in Example \ref{example}.
Once we have investigated the behavior of sample paths as $t\to\infty$, it would be natural
to consider the laws of large numbers. 
Therefore, in Section 5, we discuss the laws of large
numbers for semi-L\'evy processes.
\vskip 5mm


\section{Semi-random walks}

Crucial step in proving the recurrence
and transience dichotomy of L\'evy processes is the recurrence and transience dichotomy of
random walks.
In order to prove Theorem \ref{sLdich}, we follow this idea, and, in this section, we
introduce the notion of semi-random walk and its recurrence and transience properties.

\begin{defn}
Fix $p\in \N$.
A sequence of $\rd$-valued random variables $\{S_n\}_{n\in\Z_+}$ is called a semi-random walk with 
period $p\in\N$, if $S_0=0$ a.s.,
$
S_n-S_m\,\,\text{and}\,\, S_k-S_l\,\,\text{are independent for any }n,m,k,l\in\Z_+
\text{ satisfying }n>m\ge k>l,
$
and 
$$
S_{n+p}-S_{m+p} \eqd S_n-S_m \q \text{for any}\,\, n,m\in \Z_+.
$$
\end{defn}
Let us remark that if $\{S_n\}_{n\in\Z_+}$ is a semi-random walk with period $p$, then 
$\{S_{np}\}_{n\in\Z_+}$ is a random walk.
\begin{defn}\label{semirw}
Let $m\in\Z_+$. A semi-random walk $\{S_n\}_{n\in\Z_+}$ on $\R^d$ 
is called \emph{$m$-recurrent} if
\begin{equation*}
\liminf_{n \to \infty} |S_n-S_m| = 0 \quad \text{a.s.},
\end{equation*}
and it is called \emph{recurrent} if it is $m$-recurrent for any $m\in\Z_+$.
It is called \emph{transient} if
\begin{equation*}
\lim_{n \to \infty} |S_n| = \infty \quad \text{a.s.}
\end{equation*}
\end{defn}
\begin{rem}
Note that, for random walks, by the stationary increment property, recurrence and 
$0$-recurrence are equivalent.
In Theorem \ref{semirwdich} (i) below we show that this is
also the case for semi-random walks.
\end{rem}

The following is an example of semi-random walk, which is constructed from
a semi-L\'evy process.

\begin{ex}
Suppose that $\{X_t\}_{t\ge 0}$ is a semi-L\'evy process with period $p>0$.
Let $h\in (0,\infty)\cap p\Q$.
Then $\{X_{nh}\}_{n\in\Z_+}$ is a semi-random walk with period $n_2\in \N$, 
where $n_2$ is such that $h=pn_1/n_2$, ($n_1, n_2\in \N$).
Indeed, for any $n,m\in\Z_+$,
$$
X_{(n+n_2)h} - X_{(m+n_2)h}
= X_{nh+pn_1} - X_{mh+pn_1} 
\eqd X_{nh} - X_{mh},
$$
since $\{X_t\}_{t\ge 0}$ is a semi-L\'evy process with period $p$.
\end{ex}

\begin{thm} \label{semirwdich}
Let $\{S_n\}_{n\in\Z_+}$ be a semi-random walk on $\mathbb{R}^d$ with period $p\in\N$. 
Then the following hold.
\begin{enumerate}[\quad \rm (i)]
\item It is recurrent if it is $0$-recurrent.
\item It is either recurrent or transient.
\item It is recurrent if and only if
\begin{equation}\label{semirec}
\sum^\infty_{n=1} P(S_n \in B_a) = \infty \quad \text{for every }a>0.
\end{equation}
\item It is transient if and only if
\begin{equation}
\sum^\infty_{n=1} P(S_n \in B_a) < \infty \quad \text{for every }a>0. \label{semitran}
\end{equation}
\item It is recurrent if and only if the random walk $\{S_{np}\}_{n\in\Z_+}$ is recurrent.
\end{enumerate}
\end{thm}

In order to prove Theorem \ref{semirwdich}, we need the following lemma.

\begin{lem}\label{semi}
Let $\{S_n\}_{n\in\Z_+}$ be a semi-random walk on $\mathbb{R}^d$ with period $p\in\N$. 
The following three statements are equivalent.
\begin{enumerate}[\quad \rm (1)]
\item \eqref{semitran} holds.
\item $\{S_n\}_{n\in\Z_+}$ is transient.
\item $\{S_{np}\}_{n\in\Z_+}$ is transient.
\end{enumerate}
\end{lem}

\begin{proof}	
(1)$\Rightarrow$(2). 
Suppose \eqref{semitran}. 
Then by the Borel-Cantelli lemma, we have\\
$P(\limsup_{n \to \infty} \{|S_n| <a \}) = 0$,
namely, $P(\text{$\exists m$ such that $|S_n| \ge a$ for all $n \ge m$})=1$.
Since $a$ is arbitrary, $\{S_n\}_{n\in\Z_+}$ is transient.
	
(2)$\Rightarrow$(3). 
This follows from the definition of transience of semi-random walk in Definition \ref{semirw},
since $\{S_{np}\}_{n\in\Z_+}$ is a subsequence of $\{S_n\}_{n\in\Z_+}$.
	
(3)$\Rightarrow$(1).
Suppose
$$
 \sum^\infty_{n=1} P(S_n \in B_a) = \infty \quad \text{for some} \;\, a>0.
$$
It follows that
$$
\sum^{p-1}_{l=0} 
\sum^\infty_{n=0} P(S_{np+l} \in B_a) 
=\sum^\infty_{n=0} P(S_n \in B_a) = \infty. 
$$
Hence there is a number $l_1 \in \{0,1,\dots,p-1\}$ such that
$$
 \sum^\infty_{n=0} P(S_{np+l_1} \in B_a) = \infty. 
$$
Let $K=\{x \in \mathbb{R}^d : |x| \le a\}$. Since $B_a \subset K$,
$$
\sum^\infty_{n=0} P(S_{np+l_1} \in K) \ge \sum^\infty_{n=0} P(S_{np+l_1} \in B_a) = \infty.
$$
Fix $\eta >0$ arbitrarily. Since $K$ is compact, it is covered by a finite number of open balls 
with radius $\eta/2$. Hence, there is an open ball $B$ with radius $\eta/2$ such that
\begin{equation}
\sum^\infty_{n=0} P(S_{np+l_1} \in B) = \infty. \label{potential1}
\end{equation}
We have
\begin{align*}
1&\ge P\Biggl( \bigcup^\infty_{k=0} \bigl\{S_{kp+l_1} \in B, \; S_{(k+n)p+l_1} 
\notin B \,\; \text{for all} \; n\in\N \bigr\} \Biggr)\\
&=\sum^\infty_{k=0} P(S_{kp+l_1} \in B, \; S_{(k+n)p+l_1} \notin B \,\; 
\text{for all} \; n\in\N ) \\
&\ge \sum^\infty_{k=0} P(S_{kp+l_1} \in B, \; |S_{(k+n)p+l_1} - S_{kp+l_1}| \ge \eta 
\,\; \text{for all} \; n\in\N)\\
&=P(|S_{np+l_1} - S_{l_1}| \ge \eta \,\; \text{for all} \; n\in\N) 
\sum^\infty_{k=0} P(S_{kp+l_1} \in B)
\end{align*}
by the definition of semi-random walk. By (\ref{potential1}), we have
\begin{equation}\label{eta'}
P(|S_{np+l_1} - S_{l_1}| \ge \eta \,\; \text{for all} \; n\in\N) =0.
\end{equation}
Since $S_{p+l_1}-S_{p}$, $S_{p}-S_{l_1}$ and $S_{l_1}$ are independent and
$S_{p+l_1}-S_{p}\eqd S_{l_1}$,
we have
\begin{align*}
S_{p+l_1} - S_{l_1} &= (S_{p+l_1}-S_{p})+(S_{p}-S_{l_1})  \\
&\eqd S_{l_1}+(S_{p}-S_{l_1})=S_{p}.
\end{align*}
Since $\{S_{np+l_1} - S_{l_1}\}_{n\in\Z_+}$ and $\{S_{np}\}_{n\in\Z_+}$ are random walks
with the same law at $n=1$, we have $\{S_{np+l_1} - S_{l_1}\}_{n\in\Z_+}\eqd \{S_{np}\}_{n\in\Z_+}$.
Therefore, by \eqref{eta'},
\begin{equation}\label{eta}
P(|S_{np}| \ge \eta \,\; \text{for all} \; n\in\N) =0.
\end{equation}
Then for $k\in\Z_+$ and for $0<\eta<\ep$,
\begin{align*}
&P(|S_{kp}| < \ep - \eta, \; |S_{(k+n)p}| \ge \ep \;\, \text{for all} \; n\in\N)\\
&\qquad\le P(|S_{kp}|< \ep -\eta, \; |S_{(k+n)p}-S_{kp}| \ge 
\eta \,\; \text{for all} \; n\in\N)\\
&\qquad=   P(|S_{kp}|< \ep -\eta) P(|S_{np}| \ge \eta \,\; \text{for all} \; n\in\N) =0,
\end{align*}
by the definition of semi-random walk and \eqref{eta}. 
Letting $\eta \downarrow 0$, we have 
\begin{equation}\label{k>=1}
P(|S_{kp}| < \ep, \; |S_{(k+n)p}| \ge \ep \,\; \text{for all} \; n\in\N)=0
\end{equation}
for all $k \in\Z_+$.
Using \eqref{k>=1}, we have
\begin{align*}
P(\exists m \in\N \;\, & \text{such that} \; |S_{np}|\ge\ep \,\; \text{for all} \; n \ge m )\\
&= \sum^\infty_{k=0}  P(|S_{kp}| < \ep, \; |S_{(k+n)p}| \ge \ep \,\; \text{for all} \; n\in\N)=0,
\end{align*}
namely, $P( \forall m \in\N , \; \exists n \ge m \,\;\text{such that} \; |S_{np}| < \ep ) = 1$.
Hence $\{S_{np}\}_{n\in\Z_+}$ is $0$-recurrent, implying that $\{S_{np}\}_{n\in\Z_+}$ is not transient.
This completes the proof of the equivalence of (1)--(3). 
\end{proof}

\begin{proof}[Proof of Theorem \ref{semirwdich}]
(i) Suppose that $\{S_n\}_{n\in\Z_+}$ is not $0$-recurrent.
Then the subsequence $\{S_{np}\}_{n\in\Z_+}$ is not $0$-recurrent.
Then by Theorem 35.3 (i) of \citet{Sato's_book1999}, $\{S_{np}\}_{n\in\Z_+}$ is transient.
By the equivalence of (2) and (3) of Lemma \ref{semi}, $\{S_n\}_{n\in\Z_+}$ is transient.
Hence $\{S_n\}_{n\in\Z_+}$ is either $0$-recurrent or transient.
If $\{S_n\}_{n\in\Z_+}$ is $0$-recurrent, then, for every $m\in\Z_+$,
\begin{equation}\label{0-rec}
\liminf_{n\to\infty} |S_{n+m}-S_m|
\leq \liminf_{n\to\infty} |S_{n+m}|+|S_m|
=|S_m|<\infty\quad\text{a.s.}
\end{equation}
Since $\{S_{n+m}-S_m\}_{n\in\Z_+}$ is a semi-random walk with period $p$,
it is either $0$-recurrent or transient.
By \eqref{0-rec}, we have $\liminf_{n\to\infty} |S_{n+m}-S_m|=0$ a.s.
Thus $\{S_n\}_{n\in\Z_+}$ is recurrent.

(ii) This is proved in the proof of (i).

(iv) The claim follows from the equivalence of (1) and (2) of Lemma \ref{semi}.

(iii) Suppose $\sum^\infty_{n=1} P(S_n \in B_a) < \infty$ for some $a>0$.
Then by the Borel-Cantelli lemma, 
$P( \exists m \,\; \text{such that} \; |S_n| \ge a \;\, \text{for all} \; n \ge m) = 1$,
implying that $\{S_n\}_{n\in\Z_+}$ is not recurrent.
Next suppose \eqref{semirec}.
Then by (iv) and (ii), $\{S_n\}_{n\in\Z_+}$ is recurrent.

(v) The claim follows from (ii)
and the equivalence of (2) and (3) of Lemma \ref{semi}.
\end{proof}

\vskip 5mm

\section{Proofs of main results}
In this section, we prove Theorems \ref{sLdich} and \ref{LevsemiLev}.
We start with the following lemma, which is a generalization of Lemma 35.5 of \citet{Sato's_book1999}.

\begin{lem}\label{lemmaking}
For any semi-L\'evy process $\{X_t\}_{t\ge 0}$ and any $a>0$, there is a function 
$\gamma(a,\varepsilon)$ satisfying $\gamma(a,\varepsilon)\to1$ as $\varepsilon\downarrow 0$, 
such that,
for every $t>0$ and $\varepsilon>0$,
$$
P\left(\int_t^\infty \1_{B_{2a}}(X_s)ds>\varepsilon\right)
\geq \gamma(a, \varepsilon)P\left(|X_{t+s}|<a\text{ for some }s>0\right).
$$
\end{lem}
\begin{proof}
Let $\mathcal F_t$ be the $\sigma$-field generated by $\{X_s\}_{s\in[0,t]}$
and take $\Lambda\in\mathcal F_t$ such that
$\Lambda\subset \{|X_t|<a\}$ and $P(\Lambda)>0$.
Consider
$$
Y:=\frac 1{2\varepsilon}\int_t^{t+2\varepsilon}\1_{B_{2a}}(X_s)ds.
$$
Then $0\leq Y\leq 1$ and $2E[Y|\Lambda]\leq P(Y>1/2\,|\,\Lambda)+1$.
Hence, we have
\begin{align*}
P&\left(\int_t^{t+2\varepsilon} \1_{B_{2a}}(X_s)ds>\varepsilon\biggm| \Lambda\right)\\
& \hskip 10mm \geq \frac 1\varepsilon E\left[\int_t^{t+2\varepsilon}
\1_{B_{2a}}(X_s)ds\biggm| \Lambda\right]-1\\
& \hskip 10mm =2 \int_0^1P\left(|X_{t+2\varepsilon s}|<2a\,|\,\Lambda\right)ds-1\\
& \hskip 10mm \geq2 \int_0^1P\left(|X_{t+2\varepsilon s}-X_t|<a\,|\,\Lambda\right)ds-1\\
& \hskip 10mm =2 \int_0^1P\left(|X_{t+2\varepsilon s}-X_t|<a\right)ds-1\\
& \hskip 10mm =2\int_0^1P\left(|X_{t-[t/p]p+2\varepsilon s}-X_{t-[t/p]p}|<a\right)ds-1\\
& \hskip 10mm \geq2\int_0^1\inf_{t\in[0,p)}P\left(|X_{t+2\varepsilon s}-X_t|
<a\right)ds-1=:\gamma(a,\varepsilon),
\end{align*}
where $[x]$ denotes the largest integer not greater than $x\in\R$.
Note that a stochastically continuous process is uniformly stochastically continuous 
on any compact interval $[0,t_0]$,
i.e., $\lim_{\delta\downarrow 0}\sup_{u,v\in[0,t_0],|u-v|< \delta}P(|X_u-X_v|>a)=0$ 
for each $a>0$ (see, e.g., Lemma 9.6 of \citet{Sato's_book1999}).
Hence, as $(0,1)\ni\delta\downarrow 0$,
$$
\inf_{t\in[0,p)}P\left(|X_{t+\delta}-X_t|<a\right)\geq 
\inf_{u,v\in[0,p+1],|u-v|\leq \delta}P(|X_u-X_v|<a)\to 1.
$$
Therefore $\lim_{\delta\downarrow 0}\inf_{t\in[0,p)}P\left(|X_{t+\delta}-X_t|<a\right)=1$.
Using the bounded convergence theorem, we thus have 
$\lim_{\varepsilon\downarrow 0}\gamma(a,\varepsilon)=1$.
Hence
\begin{equation}\label{ineq}
P\left(\Lambda\cap\left\{\int_t^\infty\1_{B_{2a}}(X_s)ds>
\varepsilon\right\}\right)\geq\gamma(a,\varepsilon)P(\Lambda)
\end{equation}
for any $\Lambda\in\mathcal F_t$ satisfying
$\Lambda\subset \{|X_t|<a\}$.
For each $k\in\N$,
\begin{align*}
P&\left(\int_t^\infty\1_{B_{2a}}(X_s)ds>\varepsilon\right)\\
&\geq\sum_{n=0}^\infty P\left(|X_{t+2^{-k}j}|\geq a\text{ for }0\leq j<n, 
|X_{t+2^{-k}n}|<a,\int_{t+2^{-k}n}^\infty\1_{B_{2a}}(X_s)ds>\varepsilon\right)\\
&\geq\gamma(a,\varepsilon)\sum_{n=0}^\infty P\left(|X_{t+2^{-k}j}|\geq a\text{ for }0\leq j<n, 
|X_{t+2^{-k}n}|<a\right)\\
&=\gamma(a,\varepsilon)P\left(|X_{t+2^{-k}n}|<a\text{ for some }n\in\Z_+\right),
\end{align*}
where we have used the inequality \eqref{ineq}. 
Letting $k\to\infty$, we have
\begin{align*}
P\left(\int_t^\infty\1_{B_{2a}}(X_s)ds>\varepsilon\right)
&\geq \gamma(a, \varepsilon)P\left(|X_{t+2^{-k}n}|<a\text{ for some }k,n\in\Z_+\right)\\
&= \gamma(a, \varepsilon)P\left(|X_{t+s}|<a\text{ for some }s>0\right),
\end{align*}
where we have used the denseness of $\{2^{-k}n\colon k,n\in\Z_+\}$ in $[0,\infty)$ 
and the right-continuity of the sample paths of $\{X_t\}_{t\ge 0}$.
\end{proof}

We also need the following lemma.

\begin{lem}\label{semiLevylemma}
Let $\{X_t\}_{t\ge 0}$ be a semi-L\'evy process on $\rd$ with period $p>0$.
Fix $a>0$. Then the following statements are equivalent.
\begin{enumerate}[\quad \rm (1)]
\item $\{X_t\}_{t\ge 0}$ is $0$-recurrent.
\item $\int^\infty_0 \1_{B_a}(X_t)dt = \infty$ a.s.
\item $\int^\infty_0 P(X_t \in B_a)dt = \infty.$
\item There exists $h_0>0$ such that, for any $h\in(0,h_0]\cap p\Q$, 
the semi-random walk $\{X_{nh}\}_{n\in\Z_+}$ is recurrent.
\end{enumerate}
\end{lem}
\begin{proof}
(1)$\Rightarrow$(2). Fix $\ep>0$. 
By (1), for every $t>0$, $ P(|X_{t+s}|<a/2 \;\,\text{for some}\; s>0) = 1$. 
Using Lemma \ref{lemmaking}, we have for any $\varepsilon >0$
\begin{align*}
&P\biggl(\int^\infty_0 \1_{B_a}(X_t)dt = \infty\biggr) 
\ge P\biggl(\int^\infty_{n} \1_{B_a}(X_t)dt > 
\ep \;\,\text{for all}\; n \in \mathbb{N} \biggr) \\
&\qquad= \lim_{n\to\infty}P\biggl(\int^\infty_{n} \1_{B_a}(X_t)dt > \ep \biggr) 
\ge \gamma(a/2,\ep).
\end{align*}
Then, letting $\ep \downarrow  0$, we have (2).

(2)$\Rightarrow$(3). 
Obvious by taking expectation.

(3)$\Rightarrow$(4). 
Recall again that a stochastically continuous process is uniformly stochastically continuous 
on any compact interval.
Hence there is $h_0>0$ such that, for $u,v\in[0,p+1]$ satisfying $|u-v|<h_0$, $P(X_u-X_v\in B_a)>1/2$.
Without loss of generality, we may assume $h_0<1$.
Let $h\in (0,h_0]\cap p\Q$. Let $n\in\N$, $(n-1)h\leq t\leq nh$ and $x\in B_a$.
Then $t-[t/p]p\in [0,p)$, $nh-[t/p]p= (nh-t)+(t-[t/p]p)\in [0,h+p)\subset [0,p+1)$ 
and $|nh-t|\leq h\leq h_0$.
Therefore
\begin{align*}
P(x+X_{nh}&-X_t \in B_{2a})
\geq P(X_{nh}-X_t\in B_a)\\
&= P(X_{nh-[t/p]p}-X_{t-[t/p]p}\in B_a)> 1/2.
\end{align*}
Thus
\begin{align*}
P(X_{nh} &\in B_{2a})
\geq P(X_t\in B_a,X_t+(X_{nh}-X_t)\in B_{2a})\\
&= E\left[\1_{B_a}(X_t)\cdot P(x+X_{nh}-X_t\in B_{2a})\big|_{x=X_t}\right]
\geq \frac12 P(X_t\in B_a).
\end{align*}
Hence we have
$$
P(X_{nh}\in B_{2a})\geq \frac 1{2h}\int_{(n-1)h}^{nh} P(X_t\in B_a)dt.
$$
Thus (3) implies that $\sum_{n=1}^\infty P(X_{nh}\in B_{2a})=\infty$.
By Theorem \ref{semirwdich} (ii) and (iv), we have (4).

(4)$\Rightarrow$(1). 
If (4) holds, then 
\begin{equation}\label{limitinf}
\liminf_{t\to\infty}|X_t|\leq \liminf_{n\to\infty}|X_{nh}|=0\ \text{a.s.}
\end{equation}
Hence (1) holds.
\end{proof}	

\begin{rem}
Note that each of the statements (2) and (3) in Lemma \ref{semiLevylemma} holds for 
some $a>0$ if and only if it holds for every $a>0$.
This follows from the independence of the statements (1) and (4) from $a>0$.
Indeed, if (2) [resp.\ (3)] holds for some $a>0$, then (1) holds, implying 
that (2) [resp.\ (3)] holds for every $a>0$.
\end{rem}

\begin{proof}[Proof of Theorem \ref{sLdich}]
(v) If $\{X_t\}_{t\ge 0}$ is transient, then it is not $0$-recurrent 
and hence \eqref{sLiv} holds by
the equivalence of (1) and (3) of Lemma \ref{semiLevylemma}.
Conversely, assume \eqref{sLiv}. For $a>0$, choose $\ep>0$ such that $\gamma(a,\ep) > 1/2$,
where $\gamma$ is the function in Lemma \ref{lemmaking}. 
Then, for every $t>0$,
\begin{align*}
\int_t^\infty P(&X_s\in  B_{2a})ds
=E\left[\int^\infty_t \1_{B_{2a}}(X_s)ds\right]\\
&\geq \ep P\biggl( \int^\infty_t \1_{B_{2a}}(X_s)ds>\varepsilon \biggr) 
\geq \frac{\ep}{2}P(|X_{t+s}| < a \,\; \text{for some}\; s>0 ).
\end{align*}
Letting $t\to\infty$ and using \eqref{sLiv}, we have
$$
\lim_{t\to \infty} P(|X_{t+s}| <a \,\; \text{for some} \; s>0)=0.
$$
Hence
\begin{align*}
P\left(\lim_{t \to \infty} |X_t| = \infty\right)
&=P\left(\bigcap^\infty_{k=1}\bigcup^\infty_{n=1} \{|X_{n+s}| 
\geq k \;\,\text{for all}\; s>0 \}\right)\\
&=\lim_{k\to\infty}\lim_{n\to\infty}P(|X_{n+s}| 
\geq k \;\,\text{for all}\; s>0)=1,
\end{align*}
that is, $\{X_t\}_{t\ge 0}$ is transient.
	
(i) Assume that $\{X_t\}_{t\ge 0}$ is not $0$-recurrent. 
Then, by the equivalence of (1) and (3) of Lemma \ref{semiLevylemma},
(\ref{sLiv}) holds. 
Thus $\{X_t\}_{t\ge 0}$ is transient by (v).
Hence $\{X_t\}_{t\ge 0}$ is either $0$-recurrent or transient.
If $\{X_t\}_{t\ge 0}$ is $0$-recurrent, then, for every $s\geq0$,
\begin{equation}\label{0-rec semiLevy}
\liminf_{t\to\infty} |X_{t+s}-X_s|
\leq \liminf_{t\to\infty} |X_{t+s}|+|X_s|
=|X_s|<\infty\quad\text{a.s.}
\end{equation}
Since $\{X_{t+s}-X_s\}_{t\ge 0}$ is a semi-L\'evy process with period $p$,
it is either $0$-recurrent or transient.
By \eqref{0-rec semiLevy}, we have $\liminf_{t\to\infty} |X_{t+s}-X_s|=0$ a.s.
Thus $\{X_t\}_{t\ge 0}$ is recurrent.

(ii) The claim has been proved in (i).

(iii) By (i) and the equivalence of (1) and (3) of Lemma \ref{semiLevylemma},
we have (iii).

(iv) By (i) and the equivalence of (1) and (2) of Lemma \ref{semiLevylemma},
we have (iv).

(vi) If (\ref{sLv}) holds, then $\{X_t\}_{t\ge 0}$ is transient by (iv) and (ii).
If it is transient, then (\ref{sLiv}) holds by (v) and (\ref{sLiv}) 
implies (\ref{sLv}).
	
(vii) Let us assume that the semi-random walk $\{X_{nh}\}_{n\in\Z_+}$ is recurrent. 
Then, the claim follows from \eqref{limitinf} and (i). 
Conversely, if we suppose the recurrence of $\{X_t\}_{t\ge 0}$, 
then, by (iii), \eqref{sLii} holds. 
Fix $h\in (0,\infty)\cap p\Q$. 
Note that
$\lim_{n\to\infty}P(\sup_{s\in[0,p+h]}|X_s|<n)=P(\sup_{s\in[0,p+h]}|X_s|<\infty)=1$ 
by the right-continuity with left limits of the sample paths of $\{X_s\}_{s\geq 0}$.
Hence $P(\sup_{s\in[0,p+h]}|X_s|<a/2)>1/2$ for some $a>0$.
Let $n\in\N$, $(n-1)h\leq t\leq nh$ and $x\in B_a$.
Then $t-[t/p]p\in [0,p)$ and $nh-[t/p]p= (nh-t)+(t-[t/p]p)\in [0,h+p)$.
Therefore
\begin{align*}
P(x+X_{nh}&-X_t\in B_{2a})
\geq P(X_{nh}-X_t\in B_a)\\
&= P(X_{nh-[t/p]p}-X_{t-[t/p]p}\in B_a)\\
&\geq P(|X_{nh-[t/p]p}|<a/2,|X_{t-[t/p]p}|<a/2)\\
&\geq P\left(\sup_{s\in[0,p+h]}|X_s|<a/2\right)>1/2.
\end{align*}
Recalling the proof that (3) implies (4) in Lemma \ref{semiLevylemma},
we have $\sum_{n=1}^\infty P(X_{nh}\in B_{2a})=\infty$.
Thus, by Theorem \ref{semirwdich} (ii) and (iv), $\{X_{nh}\}_{n\in\Z_+}$ is recurrent.
\end{proof}

Using Theorem \ref{sLdich}, we prove Theorem \ref{LevsemiLev}.

\begin{proof}[Proof of Theorem \ref{LevsemiLev}]
Let $\mu_t=\law (X_t)$ and $\mu_{s,t} = \mathcal{L}(X_t-X_s)$ for $0 \leq s \le t$.

(i) Let us prove the ``only if'' part. 
By Theorem \ref{sLdich} (vii), if $\{X_t\}_{t\ge 0}$ is recurrent, 
then the random walk $\{X_{np}\}_{n\in\Z_+}$ is recurrent, and 
$\sum^\infty_{n=1} P(X_{np} \in B_a) =\infty$ for every $a>0$ by 
Theorem 35.3 (ii) of \citet{Sato's_book1999}. 
For every $n\in\N$,  
$$
\mathcal{L}(X_{np}) = \mu_{np} = \mu_{p}*\mu_{p,2p}*\dots*\mu_{(n-1)p,np} 
= \mu_p^n = \mathcal{L}(Y_n)
$$
by the definition of semi-L\'evy process. 
Hence, we also have
$\sum^\infty_{n=1} P(Y_n \in B_a) =\infty$ for every $a>0$ and again 
by Theorem 35.3 (ii) of \citet{Sato's_book1999},
the random walk $\{Y_n\}_{n\in\Z_+}$ is recurrent.
Thus $\{Y_t\}_{t\ge 0}$ is recurrent
by Theorem 35.4 (vi) of \citet{Sato's_book1999}. 	
The ``if'' part can be proved in the same way.

(ii) By Theorem \ref{sLdich} (ii), (ii) follows from (i).
\end{proof}
		
\vskip 5mm

\section{An example}

Recall the example of semi-L\'evy process defined in Example \ref{example}.
Let $\{X_t\}_{t\ge 0}$, $\{Y_t\}_{t\ge 0}$ and $\{Z_t\}_{t\ge 0}$ be as in Example \ref{example}.
Then the following holds.
\begin{prop}\label{example prop}
The recurrence of the semi-L\'evy process $\{X_t\}_{t\ge 0}$ is equivalent to that of
the L\'evy process $\{Y_{qt}+Z_{(p-q)t}\}_{t\ge 0}$.
\end{prop}
\begin{proof}
By Theorem \ref{LevsemiLev}, the recurrence of semi-L\'evy process $\{X_t\}_{t\ge 0}$ is reduced to
that of the L\'evy process $\{Y_{qt}+Z_{(p-q)t}\}_{t\ge 0}$, since $X_p=Y_q+Z_p-Z_q$.
\end{proof}

In the rest of this section we assume $d=1$. 
Further, recall that if $\{L_t\}_{t\ge 0}$ is a L\'evy process on $\R$ with $E[|L_1|]<\infty$, then 
\begin{equation}\label{Sato}
\{L_t\}_{t\ge 0} \text{ is recurrent if and only if } E[L_1]=0 
\end{equation}
(see Theorem 36.7 of \citet{Sato's_book1999}).
Now, as a consequence of Proposition \ref{example prop} and \eqref{Sato}, the following holds.

\begin{prop}\label{example d=1}
Assume $E[|Y_1|]<\infty$ and $E[|Z_1|]<\infty$.
Then $\{X_t\}_{t\ge 0}$ is recurrent 
if and only if $E[Y_q+Z_{p-q}]=0$.
\end{prop}

Due to Proposition \ref{example d=1} and \eqref{Sato}, we have the following.

\begin{cor}\label{co}
Assume $E[|Y_1|]<\infty$ and $E[|Z_1|]<\infty$.
\begin{enumerate}[\rm (i)]
\item $\{Y_t\}_{t\ge 0}$ and $\{Z_t\}_{t\ge 0}$ are recurrent $\Leftrightarrow$
$E[Y_1]=E[Z_1]=0$ $\Rightarrow$  $E[Y_q+Z_{p-q}]=0$ $\Leftrightarrow$
$\{X_t\}_{t\ge 0}$ is recurrent.
\item $\{Y_t\}_{t\ge 0}$ is recurrent and $\{Z_t\}_{t\ge 0}$ is transient $\Leftrightarrow$ 
$E[Y_1]=0$ and $E[Z_1]\ne 0$ $\Rightarrow$  $E[Y_q+Z_{p-q}]\ne 0$ $\Leftrightarrow$
$\{X_t\}_{t\ge 0}$ is transient.
\item $\{Y_t\}_{t\ge 0}$ and $\{Z_t\}_{t\ge 0}$ are transient $\Rightarrow$
$\{X_t\}_{t\ge 0}$ is recurrent if $E[Y_q] = -E[Z_{p-q}]$and it is
transient if $E[Y_q] \ne -E[Z_{p-q}]$.
\end{enumerate}
\end{cor}

\begin{rem}
\begin{enumerate}[\rm (i)]

\item
Let us remark that in the case when $E[|Y_1|] = \infty$, 
(ii) of Corollary \ref{co}  does not have to hold. 
For example, let $\{Y_t\}_{t\ge 0}$ be a symmetric Cauchy process
and let $\{Z_t\}_{t\ge 0}$ be an arbitrary L\'evy process with $E[|Z_1|] < \infty$. 
Recall that $\{Y_t\}_{t\ge 0}$ is
recurrent (see Example 35.7 of \citet{Sato's_book1999}) and $E[|Y_1|] = \infty$. 
Now, by Exercise 39.8 (i) and
(iv) of \citet{Sato's_book1999} and Proposition 4.1, $\{X_t\}_{t\ge 0}$ is recurrent.
\item
A simple example which satisfies Corollary \ref{co} (iii) is when $Y_t = Z_t = t$, in the transient
case, and $Y_t = (p-q)t$ and $Z_t = -qt$, in the recurrent case.
\end{enumerate}
\end{rem}

\vskip 5mm
\section{Laws of large numbers}
We conclude this paper by discussing laws of large numbers for semi-L\'evy processes.
We need the following lemma.

\begin{lem}\label{Doob}
Let $\{X_t\}_{t\ge 0}$ be an additive process on $\R$ and let $T>0$.
If $E[|X_t|]<\infty$ and $E[X_t]=0$ for all $t\leq T$, then
$E[\sup_{t\in[0,T]}|X_t|]\leq 8E[|X_T|]$.
\end{lem}

\begin{proof}
This lemma follows from Theorem 5.1 in Chapter VII of \citet{Doob1953} and the right continuity
of the sample paths of additive processes.
\end{proof}

\begin{thm}[Strong law of large numbers]
Let $\{X_t\}_{t\ge 0}$ be a semi-L\'evy process on $\rd$ with period $p>0$.
\begin{enumerate}[\rm (i)]
\item If $E[|X_p|]<\infty$, then $\lim_{t\to\infty}t^{-1}X_t=p^{-1}E[X_p]$ a.s.
\item If $E[|X_p|]=\infty$, then $\limsup_{t\to\infty}t^{-1}|X_t|=\infty$ a.s.
\end{enumerate}
\end{thm}

\begin{proof}
(i) Obviously, it is sufficient to prove the assertion
for each component of the process $\{X_t\}_{t\ge 0}$ on $\rd$, and hence we assume $d=1$.
We use the L\'evy-Khintchine representation of the characteristic function 
of $X_t$ in the following form:
$$
E[e^{izX_t}]=\exp\left\{-\frac12a_tz^2+i\gamma_tz
+\int_\R\left(e^{izx}-1-izx\1_{[-1,1]}(x)\right)\nu_t(dx)\right\},\ z\in\R,
$$
where $a_t\geq 0$, $\gamma_t\in\R$ and $\nu_t$ is the L\'evy measure of $X_t$ for $t\geq 0$.
Note that, if $E[|X_p|]<\infty$, then $E[|X_t|]<\infty$ for all $t\geq 0$.
Indeed we have $\int_{|x|>1}|x|\nu_t(dx)\leq \int_{|x|>1}|x|\nu_p(dx)<\infty$ for $t\leq p$,
because $\nu_t(B)\leq \nu_p(B),B\in\B(\R),$ for $t\leq p$ by Theorem 9.8 of
\citet{Sato's_book1999}, and we have
$E[|X_t|]<\infty$ for $t\leq p$.
Hence, for $t>0$, 
$E[|X_t|]=E[|X_{[t/p]p+(t-[t/p]p)}|]\leq [t/p]E[|X_p|]+E[|X_{t-[t/p]p}|]<\infty$
by Proposition 2.1 of \citet{MaejimaSato2003}.
Let $Y_n=\sup_{t\in[0,p]}|X_{t+np}-X_{np}|$.
Since $\{X_{t+np}-X_{np}\}_{t\ge 0}$ is a right-continuous process having the same
finite-dimensional distributions as $\{X_t\}_{t\ge 0}$ for each $n\in\N$,
we have $Y_n\eqd \sup_{t\in[0,p]}|X_t|$ for every $n\in\N$.
By the independent increment property of $\{X_t\}_{t\ge 0}$, $Y_n,n\in\N,$ are independent.
Also, we have $E[X_t]=\int_{|x|>1}x\nu_t(dx)+\gamma_t$.
Once again by Theorem 9.8 of \citet{Sato's_book1999}, $t\mapsto \gamma_t$ is continuous,
$\nu_s(B)\leq \nu_t(B),B\in\B(\R),$ for $s\leq t$,
and $\nu_s(B)\to\nu_t(B)$ as $s\to t$ for $B\in\B(\R)$ satisfying
$B\subset \{x\colon |x|>\varepsilon\}$ with some $\varepsilon>0$.
By the Radon-Nikodym theorem, for $t\leq t_0$, 
$\nu_t(dx)=g_t(x)\nu_{t_0}(dx)$ with some nonnegative-valued measurable function $g_t$.
This $g_t$ satisfies that $g_s\leq g_t\leq 1$ $\nu_{t_0}$-a.e.\ for $s\leq t\leq t_0$.
Also, for any $\varepsilon>0$,
by the dominated convergence theorem, 
$\int_B\lim_{s\to t}g_s(x)\nu_{t_0}(dx)=\lim_{s\to t}\int_Bg_s(x)\nu_{t_0}(dx)
=\lim_{s\to t}\nu_s(B)=\nu_t(B)=\int_Bg_t(x)\nu_{t_0}(dx)$
for $B\in\B(\R)$ satisfying $B\subset \{x\colon |x|>\varepsilon\}$,
which yields that $\lim_{s\to t}g_s(x)=g_t(x)$ $\nu_{t_0}$-a.e.\ $x\in\R$.
Then by the dominated convergence theorem again,
for $t<t_0$,
$\lim_{(0,t_0]\ni s\to t}\int_{|x|>1}x\nu_s(dx)
=\lim_{(0,t_0]\ni s\to t}\int_{|x|>1}xg_s(x)\nu_{t_0}(dx)
=\int_{|x|>1}xg_t(x)\nu_{t_0}(dx)=\int_{|x|>1}x\nu_t(dx)$.
Thus $t\mapsto\int_{|x|>1}x\nu_t(dx)$ is continuous.
Therefore $t\mapsto E[X_t]$ is continuous.
Hence $\{X_t-E[X_t]\}_{t\ge 0}$ is an additive process with $E[X_t-E[X_t]]=0$ 
and $E[\sup_{t\in[0,p]}|X_t-E[X_t]|]\leq 8E[|X_p-E[X_p]|]\leq 16E[|X_p|]$ by Lemma \ref{Doob}.
Also, $\sup_{t\in[0,p]}|E[X_t]|<\infty$ by the continuity of $t\mapsto E[X_t]$.
Thus
\begin{align*}
E[|Y_1|]
&=E\left[\sup_{t\in[0,p]}|X_t|\right]\leq E\left[\sup_{t\in[0,p]}|X_t-E[X_t]|\right]
+\sup_{t\in[0,p]}|E[X_t]|\\
&\leq 16E[|X_p|]+\sup_{t\in[0,p]}|E[X_t]|<\infty.
\end{align*}
Applying the strong law of large numbers to $\sum_{k=1}^nY_k$,
we have $n^{-1}Y_n=n^{-1}\sum_{k=1}^nY_k-n^{-1}\sum_{k=1}^{n-1}Y_k\to E[Y_1]-E[Y_1]=0$ 
as $n\to\infty$ with probability $1$.
Also, since $\{X_{np}\}_{n\in\Z_+}$ is a random walk, $\lim_{n\to\infty}n^{-1}X_{np}=E[X_p]$ a.s.
Then we have, almost surely,
$$
\left|\frac{X_t}t-\frac{E[X_p]}p\right|\leq 
\left|\frac{X_{[t/p]p}}{[t/p]}\frac{[t/p]}{t}-\frac{E[X_p]}p\right|
+\frac{Y_{[t/p]}}{[t/p]}\frac{[t/p]}{t}\to 0
$$
as $t\to\infty$.

(ii) If $E[|X_p|]=\infty$, then $\limsup_{n\to\infty}n^{-1}|X_{np}|=\infty$ a.s.
Thus $\limsup_{t\to\infty}t^{-1}|X_t|=\infty$ a.s.
\end{proof}

Finally, we prove the weak law of large numbers for semi-L\'evy processes.

\begin{thm}[Weak law of large numbers]\label{WLLN}
Let $\{X_t\}_{t\ge 0}$ be a semi-L\'evy process on $\rd$ with period $p>0$ and let $c\in\rd$.
Then the following are equivalent.
\begin{enumerate}[\quad \rm(i)]
\item $t^{-1}X_t\to c$ as $t\to\infty$ in probability.
\item $\lim_{t\to\infty}tP(|X_p|>t)=0$ and $\lim_{t\to\infty}E[X_p \1_{(0,t]}(|X_p|)]=cp$.
\end{enumerate}
\end{thm}

\begin{proof}
(i)$\Rightarrow $(ii).
Since $\{X_{np}\}_{n\in\Z_+}$ is a random walk and (i) implies that $n^{-1}X_{np}\to cp$ 
as $n\to\infty$ in probability,
we have $\lim_{t\to\infty}tP(|X_p|>t)=0$ and $\lim_{t\to\infty}E[X_p \1_{(0,t]}(|X_p|)]=cp$
by Theorem 36.4 of \citet{Sato's_book1999}.

(ii)$\Rightarrow $(i).
The statement (ii) implies $n^{-1}X_{np}\to cp$ as $n\to\infty$ in probability.
Hence 
$$\frac{X_{[t/p]p}}{[t/p]}\frac{[t/p]}{t}\to c$$
in probability.
Also, 
$$
\frac{|X_t-X_{[t/p]p}|}t\eqd \frac{|X_{t-[t/p]p}|}t
\leq \frac{\sup_{s\in[0,p]}|X_s|}t\overset{\mathrm{a.s.}}{\to}0,
$$
implying that $t^{-1}(X_t-X_{[t/p]p})\to 0$ in probability.
Hence
$$
\frac{X_t}t=\frac{X_{[t/p]p}}{[t/p]}\frac{[t/p]}{t}+\frac{X_t-X_{[t/p]p}}t\to c
$$
in probability.
\end{proof}

\begin{rem}
Recently, the notion of mean of infinitely divisible distributions has been generalized to that 
of weak mean (see \citet{Sato2010}).
More precisely, let $\mu$ be an infinitely divisible distribution with
L\'evy measure $\nu$.   
If the limit $\lim_{a\to\infty}\int_{1<|x|\leq a}x\nu(dx)$ exists in $\rd$, then
there exists $c\in \rd$ such that $\wh\mu(z)$ can be represented as 
$$
\wh\mu (z) = \exp\left \{ -\frac12 \la z, Az\ra + \lim_{a\to\infty}
\int_{|x|\leq a} \left ( e^{i\la z,x\ra}
-1 -i\la z, x\ra \right ) \nu (dx) + i\la c ,z\ra\right \},
$$
where $A$ is the Gaussian covariance matrix of $\mu$.
The vector $c$ is called the weak mean of $\mu$.
If we use the notion of weak mean,
by Theorem 5.2 of \citet{SatoUeda2012}, the statement (ii) of Theorem \ref{WLLN} is equivalent to that
$$
\text{\(X_p\) has weak mean \(cp\) and }\lim_{t\to\infty}t\int_{|x|>t}\nu_p(dx)=0,
$$
where $\nu_p$ is the L\'evy measure of $X_p$.
\end{rem}
\vskip 5mm
\n
\textbf{Acknowledgement}

\n
The authors are grateful to Ken-iti Sato for suggesting to study the dichotomy problem
of recurrence and transience of semi-L\'evy processes.
They also thank two anonymous referees for their detailed comments, by which this
paper was improved very much.
The helpful comments by Noriyoshi Sakuma are also appreciated.


\end{document}